\documentclass[12pt]{amsart}

\usepackage[utf8]{inputenc}
\usepackage{amsmath} 
\usepackage{amsfonts}
\usepackage{amssymb}
\usepackage{mathtools}
\usepackage{tikz}
\usepackage{float}
\usepackage{ytableau}

\theoremstyle{definition}
\newtheorem{theorem}{Theorem}[section]
\newtheorem*{theorem*}{Theorem}
\newtheorem{proposition}[theorem]{Proposition}

\newtheorem{lemma}[theorem]{Lemma}
\newtheorem{corollary}[theorem]{Corollary}

\newtheorem{definition}[theorem]{Definition}
\newtheorem{example}[theorem]{Example}

\theoremstyle{remark}
\newtheorem{remark}[theorem]{Remark}

\newcommand{\set}{\text{set}}
\newcommand{\llt}{\text{LLT}}
\newcommand{\asc}{\text{asc}}
\newcommand{\syt}{\text{SYT}}
\newcommand{\N}{\mathbb N}

\makeatletter
\@namedef{subjclassname@2020}{\textup{2020} Mathematics Subject Classification}
\makeatother

\begin{document}

\title[LLT polynomials of two-headed melting lollipops into ribbon Schurs]{Expanding the unicellular LLT polynomials of two-headed melting lollipops into ribbon Schurs}

\author{Victor Wang}
\address{
 Department of Mathematics,
 Harvard University,
 Cambridge MA 02138, USA}
\email{vwang@math.harvard.edu}

%TODO change

\begin{abstract} 
We prove a simple formula expanding the unicellular LLT polynomials of a class of graphs we call two-headed melting lollipops into ribbon Schur functions. Our work extends the Schur expansion originally found for melting lollipop graphs by Huh, Nam, and Yoo.
\end{abstract}

\maketitle
\tableofcontents

% Introduction
\section{Introduction}\label{sec:intro}  
In 1997, Lascoux, Leclerc, and Thibon \cite{LLT} introduced LLT polynomials, a certain $q$-deformation of products of Schur functions. When expanded into the basis of Schur functions, it is known that all coefficients are nonnegative \cite[Corollary 6.9]{GrojHaiman}, but in general, an explicit combinatorial description of the coefficients is not known and remains a major open problem. Combinatorial descriptions for the coefficients in special cases have been studied across \cite{Blasiak, HuhNamYoo, Roberts, Tom}.

For a subclass of LLT polynomials known as unicellular LLT polynomials, the polynomials have an alternative combinatorial description in terms of colourings of unit interval graphs \cite{AP}. For the unit interval graphs known as melting lollipops, Huh, Nam, and Yoo \cite{HuhNamYoo} gave an explicit combinatorial description for the Schur expansion of the associated unicellular LLT polynomials. Though not stated in their work, Huh, Nam, and Yoo implicitly showed something stronger: that the associated unicellular LLT polynomials expand as a nonnegative sum of ribbon Schur functions. Our work in this paper describes a formula expanding the unicellular LLT polynomials associated to a larger class of unit interval graphs into ribbon Schur functions.

Our paper introduces the necessary background in Section~\ref{sec:bg}. We prove our formula into ribbon Schurs for the case of melting lollipop graphs in Section~\ref{sec:melt}, and extend it to a larger class of graphs we call two-headed melting lollipops in Section~\ref{sec:two}.

\section{Background}\label{sec:bg}
Let $[n]$ denote the set $\{1,2,\dots,n\}.$ A \textit{composition} $\alpha$ is an ordered list $\alpha_1\cdots\alpha_\ell$ of positive integers. If $\alpha_1+\cdots+\alpha_\ell=n$, write $\alpha\vDash n$ and we say that $\alpha$ is a composition of $n$. Compositions of $n$ are naturally in bijection with subsets of $[n-1]$, by considering the map defined by $\set(\alpha)=\{\alpha_1,\alpha_1+\alpha_2,\dots,\alpha_1+\dots+\alpha_{\ell-1}\}$. The \textit{concatenation} $\alpha\cdot\beta\vDash n+m$ of $\alpha=\alpha_1\dots\alpha_\ell\vDash n$ and $\beta=\beta_1\dots\beta_k\vDash m$ is the composition satisfying $\set(\alpha\cdot\beta)=\{\alpha_1,\alpha_1+\alpha_2,\dots,\alpha_1+\dots+\alpha_\ell,n+\beta_1,n+\beta_1+\beta_2,\dots,n+\beta_1+\dots+\beta_{k-1}\}$. Their \textit{near concatenation} $\alpha\odot\beta\vDash n+m$ satisfies $\set(\alpha\odot\beta)=\set(\alpha\cdot\beta)\setminus\{n\}$. The \textit{reverse} of a composition $\alpha=\alpha_1\dots\alpha_\ell$ is the composition $\alpha^r=\alpha_\ell\dots\alpha_1$. More generally, for a list of integers $\mathbf v=(v_1,\dots,v_n)$, its reverse is the list $(v_n,\dots,v_1)$. For a list of integers $\mathbf v=(v_1,\dots,v_n)$ and $S\subseteq [n]$, we write $\mathbf v(S)$ to denote the value $\sum_{s\in S}v_s$.

\begin{definition}
A \textit{unit interval graph} $G$ on $n$ vertices is a graph with vertex set $[n]$ with the property that if $(i,j)\in E(G)$ and $i\le k<l\le j$ then $(k,l)\in E(G).$
\end{definition}

When we draw a unit interval graph, we will draw its vertices in increasing order from left to right. Two unit interval graphs will encounter are the \textit{complete graph} $K_n$ on $n$ vertices with edge set $E(K_n)=\{(i,j):1\le i<j\le n\}$ and the \textit{path} $P_n$ on $n$ vertices with edge set $E(P_n)=\{(i,i+1):1\le i\le n+1)\}$. 

The \textit{reverse} $G^r$ of a unit interval graph $G$ on $n$ vertices is the unit interval graph on vertex set $[n]$ and edge set $E(G^r)=\{(n+1-j,n+1-i):(i,j)\in E(H)\}$. The \textit{concatenation} $G+H$ of two unit interval graphs $G$ and $H$ on $n$ and $m$ vertices, respectively, is the unit interval graph on vertex set $[n+m-1]$ and edge set $E(G+H)=E(G)\cup\{(i+n-1,j+n-1):(i,j)\in E(H)\}$. Their \textit{disjoint union} $G\cup H$ is the unit interval graph on the vertex set $[n+m]$ and edge set $E(G\cup H)=E(G)\cup \{(i+n,j+n):(i,j)\in E(H)\}$.

\begin{definition}
Given a unit interval graph $G$ on $n$ vertices, its \textit{area sequence} is the sequence $\mathbf a=(a_1,\dots,a_{n-1})$, where $a_i=\max(\{0\}\cup\{j-i:(i,j)\in E(G)\}).$ The \textit{transpose} of $\mathbf a$ is the area sequence $\mathbf a^T$ associated to $G^r$.
\end{definition}

\begin{example}
In the picture below, the area sequence of $G$ is $\mathbf a=(2,2,1,1)$. The area sequence of $G^r$, the reverse graph of $G$, is $\mathbf a^T=(1,2,2,1)$. Note that in general, $\mathbf a$ is not the reverse of $\mathbf a^T$.
\begin{center}
%\vspace*{.1in}
\begin{tikzpicture}

\coordinate (A) at (11.75,0);
\coordinate (B) at (12.5,0);
\coordinate (C) at (13.25,0);
\coordinate (D) at (14,0);
\coordinate (E) at (14.75,0);
\draw[black] (A)--(B)--(C)--(D)--(E);
\draw[black] (13.25,0) arc (0:180:.75);
\draw[black] (14,0) arc (0:180:.75);

\filldraw[black] (A) circle [radius=2pt];
\filldraw[black] (B) circle [radius=2pt];
\filldraw[black] (C) circle [radius=2pt];
\filldraw[black] (E) circle [radius=2pt];
\filldraw[black] (D) circle [radius=2pt];
\node [] at (13.25,-.5) {$G$};

\coordinate (A) at (16.75,0);
\coordinate (B) at (17.5,0);
\coordinate (C) at (18.25,0);
\coordinate (D) at (19,0);
\coordinate (E) at (19.75,0);
\draw[black] (A)--(B)--(C)--(D)--(E);
\draw[black] (19.75,0) arc (0:180:.75);
\draw[black] (19,0) arc (0:180:.75);

\filldraw[black] (A) circle [radius=2pt];
\filldraw[black] (B) circle [radius=2pt];
\filldraw[black] (C) circle [radius=2pt];
\filldraw[black] (E) circle [radius=2pt];
\filldraw[black] (D) circle [radius=2pt];
\node [] at (18.25,-.5) {$G^r$};
\end{tikzpicture}
\end{center}
\end{example}

\begin{definition}\cite[Definition 3.7]{AP}
Given an area sequence $\mathbf a$ associated to the unit interval graph $G$ on $n$ vertices, the \textit{unicellular LLT polynomial} of $\mathbf a$ (or, synonymously, the unicellular LLT polynomial of $G$) is defined to be
$$\llt_{\mathbf a}(\mathbf x;q)=\sum_{\kappa:[n]\to\N}q^{\asc(\kappa)}x_{\kappa(1)}\cdots x_{\kappa(n)},$$
where for each map $\kappa:[n]\to\N,$ the number $\asc(\kappa)$ counts the number of pairs $(i,j)\in E(G)$ such that $i<j$ and $\kappa(i)<\kappa(j).$
\end{definition}

LLT polynomials lie in the ring of symmetric functions, and as a consequence, $\llt_{\mathbf a}(\mathbf x;q)=\llt_{\mathbf a^T}(\mathbf x;q).$ Note also that the unicellular LLT polynomial of $G\cup H$ is simply the product of the unicellular LLT polynomials of $G$ and $H$. Unicellular LLT polynomials satisfy a recurrence first proved by Lee as \cite[Theorem 3.4]{Lee}. We will use the formulation stated in \cite{HuhNamYoo}.

\begin{theorem}\label{theorem:modular}\cite[Theorem 3.4]{HuhNamYoo}
Let area sequence $\mathbf a=(a_1,\dots,a_{n-1})$ and let $i$ be such that $a_{i-1}+1\le a_i$ (letting $a_0$ be $0$). Suppose area sequences $\mathbf a'$ and $\mathbf a''$ differ from $\mathbf a$ only in position $i$ with $a_i=a_i'+1=a_i''+2$. If $a_{i+a_i-1}=a_{i+a_i}+1$, then
$$\llt_{\mathbf a}(\mathbf x;q)+q\llt_{\mathbf a''}(\mathbf x;q)=(1+q)\llt_{\mathbf a'}(\mathbf x;q).$$
\end{theorem}

Two other classes of symmetric functions we will be interested in are ribbon Schur functions and Schur functions.

\begin{definition}
The \textit{ribbon Schur function} $r_\alpha$ for $\alpha\vDash n$ is given by
$$r_\alpha=\sum_{\substack{\kappa:[n]\to\N\\\kappa(i)<\kappa(i+1)\text{ if }i\in \set(\alpha)\\\kappa(i)\ge\kappa(i+1)\text{ if }i\not\in \set(\alpha)}}x_{\kappa(1)}\cdots x_{\kappa(n)}.$$
\end{definition}

Ribbon Schur functions are symmetric functions with $r_\alpha r_\beta=r_{\alpha\cdot\beta}+r_{\alpha\odot\beta}$ and $r_\alpha=r_{\alpha^r}$. They expand with nonnegative coefficients counting a class of combinatorial objects into functions known as Schur functions. The Schur functions $\{s_\lambda\}$ are a basis for symmetric functions indexed by integer partitions. A \textit{partition} $\lambda=\lambda_1\cdots\lambda_\ell$ is a composition with weakly decreasing parts. If $\lambda_1+\dots+\lambda_\ell=n$, we write $\lambda\vdash n$ and say that $\lambda$ is a partition of $n$.

It is an open problem to give a combinatorial formula for the expansion of unicellular LLT polynomials into Schur functions, though it is known via a representation-theoretic argument that all coefficients are nonnegative \cite[Corollary 6.9]{GrojHaiman}. Thus, if we can give a formula writing a unicellular LLT polynomial as a nonnegative sum of ribbon Schur functions, we obtain a combinatorial formula for the expansion into Schur functions.

We will conclude this section by describing the expansion of ribbon Schurs into Schurs. The \textit{Young diagram} of a partition $\lambda$ is a left-justified array of cells with $\lambda_i$ cells in the $i$th row. A \textit{standard Young tableau} of shape $\lambda$ is a filling of the Young diagram of $\lambda$ with the integers $1,\dots,n$ occurring exactly once each so that the rows and columns are increasing. Write $\syt(\lambda)$ for the set of all standard Young tableaux of shape $\lambda$. The \textit{descent set} $D(T)$ of a standard Young tableau $T$ is the subset of $[n-1]$ consisting of all the integers $i$ for which $i+1$ is in a later row than $i$.

\begin{proposition}\label{prop:ribbon}
The ribbon Schur function $r_\alpha$ for $\alpha\vDash n$ expands into the basis of Schur functions via
$$r_\alpha=\sum_{\lambda\vdash n}\sum_{\substack{T\in\syt(\lambda)\\ D(T)=\set(\alpha)}}s_\lambda.$$
\end{proposition}
\begin{example}
We have $r_{22}=s_{31}+s_{22}$ via the following standard Young tableaux $T$ with $D(T)=\{2\}$.
\ytableausetup{centertableaux}
\begin{center}
\begin{ytableau}
1 & 2 & 4 \\
3 \\
\end{ytableau}
\quad\quad\quad\quad \quad\quad
\begin{ytableau}
1 & 2 \\
3 & 4 \\
\end{ytableau}
\end{center}
\end{example}

\section{Melting lollipop graphs}\label{sec:melt}

In this section, we prove that the unicellular LLT polynomials of a class of graphs known as melting lollipop graphs expand as a nonnegative sum of ribbon Schur functions.

\begin{definition}
For $m\ge 1$ and $0\le k\le m-1$, the \textit{melting complete graph} $K_m^{(k)}$ is the unit interval graph on $m$ vertices obtained by removing the $k$ edges $(1,m),(1,m-1),\dots,(1,m-k+1)$ from the complete graph $K_m$.
\end{definition}

\begin{definition}
For $n\ge 0$, $m\ge 1$, and $0\le k\le m-1$, the \textit{melting lollipop graph} $L_{m,n}^{(k)}$ is the unit interval graph on $n+m$ vertices given by the concatenation $P_{n+1}+K_{m}^{(k)}$.
\end{definition}

\begin{example}
The graph drawn below is $L_{5,3}^{(2)}=P_4+K_5^{(2)}$.
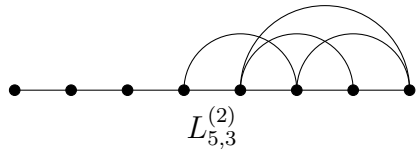
\begin{center}
\begin{tikzpicture}

\coordinate (A) at (16.75,0);
\coordinate (B) at (17.5,0);
\coordinate (C) at (18.25,0);
\coordinate (D) at (19,0);
\coordinate (E) at (19.75,0);
\coordinate (F) at (20.5,0);
\coordinate (G) at (16,0);
\coordinate (H) at (15.25,0);
\draw[black] (H)--(G)--(A)--(B)--(C)--(D)--(E)--(F);
\draw[black] (19.75,0) arc (0:180:.75);
\draw[black] (19,0) arc (0:180:.75);
\draw[black] (20.5,0) arc (0:180:.75);
\draw[black] (20.5,0) arc (0:180:1.125);

\filldraw[black] (A) circle [radius=2pt];
\filldraw[black] (B) circle [radius=2pt];
\filldraw[black] (C) circle [radius=2pt];
\filldraw[black] (E) circle [radius=2pt];
\filldraw[black] (D) circle [radius=2pt];
\filldraw[black] (F) circle [radius=2pt];
\filldraw[black] (G) circle [radius=2pt];
\filldraw[black] (H) circle [radius=2pt];
\node [] at (18.25-.375,-.5) {$L_{5,3}^{(2)}$};
\end{tikzpicture}
\end{center}    
\end{example}

Before we prove our formula for melting lollipop graphs, we describe three simple lemmas.

\begin{lemma}\label{lemma:path}
Let $\mathbf a$ be the area sequence of the path $P_n$ on $n$ vertices, i.e. $\mathbf a=(1^{n-1})$. Then
$$\llt_{\mathbf a}(\mathbf x;q)=\sum_{\alpha\vDash n}q^{\mathbf a(\set(\alpha))}r_\alpha.$$
\end{lemma}

\begin{proof}
We compute
\begin{align*}
    \llt_{\mathbf a}(\mathbf x;q)&=\sum_{\kappa:[n]\to\mathbb N}q^{\asc(\kappa)}x_{\kappa(1)}\cdots x_{\kappa(n)}\\
    &= \sum_{S\subseteq[n-1]}\sum_{\substack{\kappa:[n]\to\N\\\kappa(i)<\kappa(i+1)\text{ if }i\in S\\\kappa(i)\ge\kappa(i+1)\text{ if }i\not\in S}}q^{|S|} x_{\kappa(1)}\cdots x_{\kappa(n)}\\
    &=\sum_{\alpha\vDash n}q^{\mathbf a(\set(\alpha))}r_\alpha.
\end{align*}
\end{proof}

\begin{lemma}\label{lemma:union}
Let $\mathbf v=(v_1,\dots,v_{n-1})$ and $\mathbf w=(w_1,\dots,w_{m-1})$ be two lists of nonnegative integers. Then
$$\left(\sum_{\beta\vDash  n} q^{\mathbf v(\set(\beta))}r_\beta\right)\left(\sum_{\gamma\vDash m} q^{\mathbf w(\set(\gamma))}r_\gamma\right)=\sum_{\alpha\vDash  n+m} q^{(\mathbf v,0,\mathbf w)(\set(\alpha))}r_\alpha.$$
\end{lemma}
\begin{proof}
We compute
\begin{align*}
    \left(\sum_{\beta\vDash  n} q^{\mathbf v(\set(\beta))}r_\beta\right)\left(\sum_{\gamma\vDash m} q^{\mathbf w(\set(\gamma))}r_\gamma\right)&=\sum_{\substack{\beta\vDash n\\\gamma\vDash m}}q^{\mathbf v(\set(\beta))+\mathbf w(\set(\gamma))}(r_{\beta\cdot\gamma}+r_{\beta\odot\gamma})\\
    &=\sum_{\substack{\alpha\vDash  n+m\\ n\in\set(\alpha)}}q^{(\mathbf v,0,\mathbf w)(\set(\alpha))}r_\alpha + \sum_{\substack{\alpha\vDash  n+m\\ n\not\in\set(\alpha)}}q^{(\mathbf v,0,\mathbf w)(\set(\alpha))}r_\alpha\\
    &=\sum_{\alpha\vDash  n+m} q^{(\mathbf v,0,\mathbf w)(\set(\alpha))}r_\alpha.
\end{align*}
\end{proof}

\begin{lemma}\label{lemma:prog}
Let $\mathbf v, \mathbf v',\mathbf v''$ be three lists of nonnegative integers of length $n-1$ that differ only in position $i$, such that $v_i=v_i'+1=v_i''+2$. Then
$$\sum_{\alpha\vDash  n} q^{\mathbf v(\set(\alpha))}r_\alpha+q\sum_{\alpha\vDash  n} q^{\mathbf v''(\set(\alpha))}r_\alpha=(1+q)\sum_{\alpha\vDash  n} q^{\mathbf v'(\set(\alpha))}r_\alpha.$$
\end{lemma}
\begin{proof}
We compute
\begin{align*}
    \sum_{\alpha\vDash  n} q^{\mathbf v(\set(\alpha))}r_\alpha+q\sum_{\alpha\vDash  n} q^{\mathbf v''(\set(\alpha))}r_\alpha&=\sum_{\alpha\vDash n}(q^{\mathbf v(\set(\alpha))}+q^{\mathbf v''(\set(\alpha))+1})r_\alpha\\
    &=\sum_{\substack{\alpha\vDash n\\i\in\set(\alpha)}}(q+1)q^{\mathbf v'(\set(\alpha))}r_\alpha +\sum_{\substack{\alpha\vDash n\\i\not\in\set(\alpha)}}(1+q)q^{\mathbf v'(\set(\alpha))}r_\alpha\\
    &=(1+q)\sum_{\alpha\vDash  n} q^{\mathbf v'(\set(\alpha))}r_\alpha.
\end{align*}
\end{proof}

We now give a formula expanding the unicellular LLT polynomials of melting lollipops as a nonnegative sum of ribbon Schur functions.

\begin{theorem}\label{thm:melt}
Let $\mathbf a$ be the area sequence of the melting lollipop graph $L_{m,n}^{(k)}$ on $n+m$ vertices. Then
$$\llt_{\mathbf a}(\mathbf x;q)=\sum_{\alpha\vDash n+m}q^{\mathbf a(\set(\alpha))}r_\alpha.$$
\end{theorem}
\begin{proof}
    Our proof employs a triple induction. We first induct on the size of $n+m$.
    
    For $n+m=1$, the associated unit interval graph must be $P_1$, and the result is immediate. 
    
    For the inductive step, assume $n+m\ge 2$. We then induct on the size of $m$ for fixed $n+m$. For $m=1$, the unit interval graph $L_{m,n}^{(k)}$ is just the path $P_{n+1}$, and the result follows from Lemma~\ref{lemma:path}. 
    
    In the inductive step of the second induction, we assume $m\ge 2$. We finally induct on the size of $k$, as $k$ grows smaller from $k=m-1$ to $k=0$. The base cases in this induction are $k=m-1$ and $k=m-2$.

    When $k=m-1$, the associated melting lollipop graph $L_{m,n}^{(k)}$ is the disjoint union $P_{n+1}\cup K_{m-1}$. Note $K_{m-1}$ is itself a melting lollipop graph with area sequence $(m-2,m-3,\dots 1)$ on strictly fewer than $n+m$ vertices. Applying Lemma~\ref{lemma:path} and Lemma~\ref{lemma:union} in the case where $L_{m,n}^{(k)}=P_{n+1}\cup K_{m-1}$ (and hence has area sequence $\mathbf a=(1^n,0,m-2,m-3,\dots 1)$),
    \begin{align*}
    \llt_{\mathbf a}(\mathbf x;q)&=\left(\sum_{\beta\vDash n+1} q^{(1^n)(\set(\beta))r_\beta}\right)\left(\sum_{\gamma\vDash m-1}q^{(m-2,m-3,\dots,1)(\set(\gamma))}r_\gamma\right)\\
    &=\sum_{\alpha\vDash n+m}q^{(1^n,0,m-2,m-3,\dots,1)(\set(\alpha))}r_\alpha=\sum_{\alpha\vDash n+m}q^{\mathbf a(\set(\alpha))}r_\alpha.
    \end{align*}
    The other base case $k=m-2$ is given by the inductive hypotheses, since then $L_{m,n}^{(k)}=L_{m-1,n+1}^{(0)}$.
    
    Finally, we assume $k\le m-3$ in the inductive step of our third induction. Let $\mathbf a'$ and $\mathbf a''$ denote the area sequences of $L_{m,n}^{(k+1)}$ and $L_{m,n}^{(k+2)}$, respectively. Note $\mathbf a$, $\mathbf a'$, and $\mathbf a''$ differ only in position $n+1$ with $a_{n+1}=a_{n+1}'+1=a_{n+1}''+2$, and the hypotheses of both Theorem~\ref{theorem:modular} and Lemma~\ref{lemma:prog} are satisfied. Therefore,
    \begin{align*}
        \llt_{\mathbf a}(\mathbf x;q)&=(1+q)\llt_{\mathbf a'}(\mathbf x;q)-q\llt_{\mathbf a''}(\mathbf x;q)\\
        &=(1+q)\sum_{\alpha\vDash  n+m} q^{\mathbf a'(\set(\alpha))}r_\alpha-q\sum_{\alpha\vDash  n+m} q^{\mathbf a''(\set(\alpha))}r_\alpha\\
        &=\sum_{\alpha\vDash  n+m} q^{\mathbf a(\set(\alpha))}r_\alpha,
    \end{align*}
    where the second equality is by applying the inductive hypotheses. This completes the proof of the third inductive step, finishing the proof of the theorem.
\end{proof}

\section{Two-headed melting lollipop graphs}\label{sec:two}

In this section, we prove a nonnegative formula into ribbon Schurs for the unicellular LLT polynomials of a larger class of unit interval graphs, which we call two-headed melting lollipops.

\begin{definition}
For $n\ge -1$, $m_1\ge 1$, $0\le k_1\le m_1-1$, $m_2\ge 1$, and $0\le k_2 \le m_2-1$, the \textit{two-headed melting lollipop graph} $\prescript{(k_1)}{m_1}{L_{m_2,n}^{(k_2)}}$ is the unit interval graph on $m_1+n+m_2$ vertices given by the concatenation $(K_{m_1}^{(k_1)})^r+P_{n+2}+K_{m_2}^{(k_2)}.$
\end{definition}
\begin{example}
The graph drawn below is
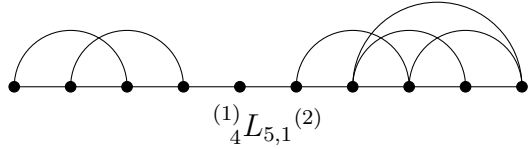
\begin{center}
\begin{tikzpicture}

\coordinate (A) at (16.75,0);
\coordinate (B) at (17.5,0);
\coordinate (C) at (18.25,0);
\coordinate (D) at (19,0);
\coordinate (E) at (19.75,0);
\coordinate (F) at (20.5,0);
\coordinate (G) at (16,0);
\coordinate (H) at (15.25,0);
\coordinate (I) at (14.5,0);
\coordinate (J) at (13.75,0);
\draw[black] (J)--(I)--(H)--(G)--(A)--(B)--(C)--(D)--(E)--(F);
\draw[black] (19.75,0) arc (0:180:.75);
\draw[black] (19,0) arc (0:180:.75);
\draw[black] (15.25,0) arc (0:180:.75);
\draw[black] (16,0) arc (0:180:.75);
\draw[black] (20.5,0) arc (0:180:.75);
\draw[black] (20.5,0) arc (0:180:1.125);

\filldraw[black] (A) circle [radius=2pt];
\filldraw[black] (B) circle [radius=2pt];
\filldraw[black] (C) circle [radius=2pt];
\filldraw[black] (E) circle [radius=2pt];
\filldraw[black] (D) circle [radius=2pt];
\filldraw[black] (F) circle [radius=2pt];
\filldraw[black] (G) circle [radius=2pt];
\filldraw[black] (H) circle [radius=2pt];
\filldraw[black] (I) circle [radius=2pt];
\filldraw[black] (J) circle [radius=2pt];
\node [] at (18.25-.375-.75,-.5) {$\prescript{(1)}{4}{L_{5,1}}^{(2)}$};
\end{tikzpicture}
\end{center}   
\end{example}
\begin{theorem}\label{thm:twoheaded}
Let $\mathbf a$ be the area sequence of the two-headed melting lollipop graph $\prescript{(k_1)}{m_1}{L_{m_2,n}^{(k_2)}}$ on $m_1+n+m_2$ vertices. Let $\mathbf b$ denote the modified sequence $(1,2,\dots,m_1-2,m_1-k_1-1,a_{m_1},a_{m_1+1},\dots ,a_{m_1+n+m_2})$. (When $m_1=1$, we understand the sequence $\mathbf b$ to just be $\mathbf a$.) Then
$$\llt_{\mathbf a}(\mathbf x;q)=\sum_{\alpha \vDash n+m}q^{\mathbf b(\set(\alpha))}r_\alpha.$$
\end{theorem}
\begin{proof}
    Our proof this time will employ a double induction. We begin by inducting on the size of $m$.

    For the base case $m_2=1$, note the two-headed melting lollipop $\prescript{(k_1)}{m_1}{L_{m_2,n}^{(k_2)}}$ is just $(L_{m_1,n+1}^{(k_2)})^r$ the reverse of the melting lollipop graph $L_{m_1,n+1}^{(k_2)}$. Thus $\mathbf a^T$ is the reverse of $\mathbf b$, with $\mathbf b$ being the reverse of the area sequence of $L_{m_1,n+1}^{(k_2)}$ in this case. The result then follows because
    $$\llt_{\mathbf a}(\mathbf x;q)=\llt_{\mathbf a^T}(\mathbf x;q)=\sum_{\alpha\vDash n+m}q^{\mathbf b(\set(\alpha))}r_{\alpha^r}=\sum_{\alpha\vDash n+m}q^{\mathbf b(\set(\alpha))}r_{\alpha},$$
    with the second equality by Theorem~\ref{thm:melt}.

    For the inductive step, assume $m_2\ge 2$. We now induct on the size of $k_2$, as $k_2$ grows smaller from $k_2=m_2-1$ to $k_2=0$. The base cases in this induction are $k_2=m_2-1$ and $k_2=m_2-2$.

    When $k_2=m_2-1$, the two-headed melting lollipop $\prescript{(k_1)}{m_1}{L_{m_2,n}^{(k_2)}}$ is $(L_{m_1,n+1}^{(k_1)})^r\cup K_{m_2-1}$. In this case, $\mathbf b=(1,2\dots m_1-2,m_1-k_1-1,1^{n+1},0,m_2-2,m_2-3,\dots,1)$. Since the LLT polynomial of $\mathbf a$ is the product of the LLT polynomials of the area sequences associated with $(L_{m_1,n+1}^{(k_1)})^r$ and $K_{m_2-1}$, we find that
    \begin{align*}
        \llt_{\mathbf a}(\mathbf x;q)&=\left(\sum_{\beta\vDash m_1+n+1}q^{(1,2\dots m_1-2,m_1-k_1-1,1^{n+1})(\set(\beta))}r_\beta\right)\left(\sum_{\gamma\vDash{m_2-1}}q^{(m_2-2,m_2-3,\dots,1)(\set(\gamma))}\right)\\
        &=\sum_{\alpha \vDash n+m}q^{\mathbf b(\set(\alpha))}r_\alpha.
    \end{align*}
    Here, the first equality applies the inductive hypotheses to give the expression for the LLT polynomial associated to $(L_{m_1,n+1}^{(k_1)})^r=\prescript{(k_1)}{m_1}{L_{1,n}^{(0)}}$ and Theorem~\ref{thm:melt} for the LLT polynomial associated to $K_{m_2-1}$. The second equality is an application of Lemma~\ref{lemma:union}.

    For the case $k_2=m_2-2$, note $\prescript{(k_1)}{m_1}{L_{m_2,n}^{(k_2)}}$ is the two-headed melting lollipop $\prescript{(k_1)}{m_1}{L_{m_2-1,n+1}^{(0)}}$, and the result follows by applying the inductive hypotheses.

    Finally, we assume $k_2\le m_2-3$ in the inductive step of the second induction. Let $\mathbf a'$ and $\mathbf a''$ denote the area sequences of $\prescript{(k_1)}{m_1}{L_{m_2,n}^{(k_2+1)}}$ and $\prescript{(k_1)}{m_1}{L_{m_2,n}^{(k_2+2)}}$, respectively, and let $\mathbf b'$ and $\mathbf b''$ be the associated sequences constructed by the statement of the theorem. Note $\mathbf a,\mathbf a',\mathbf a''$ satisfy the conditions of Theorem~\ref{theorem:modular} in position $m_1+n+1$ and $\mathbf b,\mathbf b',\mathbf b''$ satisfy the conditions of Lemma~\ref{lemma:prog} in position $m_1+n+1$. Therefore,
    \begin{align*}
    \llt_{\mathbf a}(\mathbf x;q)&=(1+q)\llt_{\mathbf a'}(\mathbf x;q)-q\llt_{\mathbf a''}(\mathbf x;q)\\
    &=(1+q)\sum_{\alpha \vDash n+m}q^{\mathbf b'(\set(\alpha))}r_\alpha-q\sum_{\alpha \vDash n+m}q^{\mathbf b''(\set(\alpha))}r_\alpha=\sum_{\alpha \vDash n+m}q^{\mathbf b(\set(\alpha))}r_\alpha.
    \end{align*}
    This completes the proof of the second inductive step, finishing the proof of the theorem.
\end{proof}

For completeness, we state an explicit combinatorial Schur expansion for the unicellular LLT polynomials of two-headed melting lollipop graphs by applying Proposition~\ref{prop:ribbon} to the statement of Theorem~\ref{thm:twoheaded}.

\begin{corollary}
Let $\mathbf a$ be the area sequence of the two-headed melting lollipop graph $\prescript{(k_1)}{m_1}{L_{m_2,n}^{(k_2)}}$ on $m_1+n+m_2$ vertices. Let $\mathbf b$ denote the modified sequence $(1,2,\dots,m_1-2,m_1-k_1-1,a_{m_1},a_{m_1+1},\dots ,a_{m_1+n+m_2})$. Then
$$\llt_{\mathbf a}(\mathbf x;q)=\sum_{\lambda \vdash m_1+n+m_2}\sum_{T\in\syt(\lambda)}q^{\mathbf b(D(T))}s_\lambda.$$
\end{corollary}
\begin{example}
Let $\mathbf a = (2,1,2,1)$, which is the area sequence of a two-headed melting lollipop. We construct the modified sequence $\mathbf b=(1,2,2,1)$. The coefficient of $s_{32}$ in the Schur expansion of $\llt_{\mathbf a}(\mathbf x;q)$ is $q^{1+2}+q^{1+1}+q^{2}+q^{2+1}+q^2=3q^2+2q^3,$ from the following standard Young tableaux.
\begin{center}
\begin{ytableau}
1 & 3 & 5 \\
2&4 \\
\end{ytableau}
\quad\quad
\begin{ytableau}
1 & 3 & 4 \\
2 & 5 \\
\end{ytableau}
\quad\quad
\begin{ytableau}
1 & 2 & 5 \\
3 & 4 \\
\end{ytableau}
\quad\quad
\begin{ytableau}
1 & 2 & 4 \\
3 & 5 \\
\end{ytableau}
\quad\quad
\begin{ytableau}
1 & 2 & 3 \\
4 & 5 \\
\end{ytableau}
\end{center}
\end{example}
\begin{remark}
When $m_1=1$, i.e. when $\prescript{(k_1)}{m_1}{L_{m_2,n}^{(k_2)}}$ is a melting lollipop graph, we recover \cite[Proposition 5.9]{HuhNamYoo}.
\end{remark}

% Ack
\section{Acknowledgments}\label{sec:ack}  The author would like to thank Per Alexandersson for helpful feedback on an early draft.
%%%%%

%: Bibliography
\bibliographystyle{plain}

\end{document}